\documentclass[a4paper,11pt,twoside,notitlepage,reqno]{amsart}

\usepackage{amsmath,amssymb,amsbsy,amsthm}
\usepackage[hmargin=1.4in,vmargin=1.4in]{geometry}
\usepackage{url}
\numberwithin{equation}{section}
\usepackage{tikz}
\usetikzlibrary{matrix}
\usepackage[linktocpage]{hyperref}
\hypersetup{
    colorlinks=true,        
    linkcolor=red,          
    citecolor=red,         
    filecolor=magenta,      
    urlcolor=cyan           
}

\addtocontents{toc}{\protect\setcounter{tocdepth}{1}}

\usepackage{eucal}

\makeatletter
\@namedef{subjclassname@2020}{%
  \textup{2020} Mathematics Subject Classification}
\makeatother

\theoremstyle{plain}
	\newtheorem{theorem}{Theorem}

		\numberwithin{theorem}{section}
	\newtheorem{lemma}[theorem]{Lemma}
	\newtheorem{fact}[theorem]{Fact}
	\newtheorem{proposition}[theorem]{Proposition}

	\newtheorem*{theorem*}{Theorem}
	\newtheorem*{lemma*}{Lemma}
	
	\newtheorem*{cor*}{Corollary}
	\newtheorem*{conj*}{Conjecture}

\theoremstyle{definition}
	
	\newtheorem*{example*}{Example}
	\newtheorem{definition}[theorem]{Definition}
	\newtheorem{remark}[theorem]{Remark}

\begin{document}

\title[Ore extensions and the Dixmier-Moeglin equivalence]{Ore extensions of commutative rings and the Dixmier-Moeglin equivalence}

\author[Jason P. Bell]{Jason P. Bell}
\address{Department of Pure Mathematics \\University of Waterloo \\ Waterloo, ON \\ Canada \\ N2L 3G1}
\email{jpbell@uwaterloo.ca}

\author[L\'eon Burkhardt]{L\'eon Burkhardt}
\address{Departement of Mathematics, ETH Zurich, R\"amistrasse 101, 8092 Zurich, Switzerland}
\email{lburkhardt@student.ethz.ch}

\author[Nicholas Priebe]{Nicholas Priebe}
\address{Department of Pure Mathematics \\University of Waterloo \\ Waterloo, ON \\ Canada \\ N2L 3G1}
\email{nicholas.priebe@uwaterloo.ca}



\keywords{primitive ideals, Dixmier-Moeglin equivalence, prime spectrum}
\subjclass[2020]{16D60, 16A20, 16A32}
\maketitle
\begin{abstract}
We consider Ore extensions of the form $T:=R[x;\sigma,\delta]$ with $R$ a commutative integral domain that is finitely generated over a field $k$.  We show that if $T$ has Gelfand-Kirillov dimension less than four then a prime ideal $P\in {\rm Spec}(T)$ is primitive if and only if $\{P\}$ is locally closed in ${\rm Spec}(T)$, if and only if the Goldie ring of quotients of $T/P$ has centre that is an algebraic extension of $k$.  We also show that there are examples for which these equivalences do not all hold for $T$ of integer Gelfand-Kirillov dimension greater than or equal to $4$.
\end{abstract}

\section{Introduction}
In the study of rings, it is of special importance to understand the irreducible representations of a ring; that is, to understand its simple left or right modules.  This approach is especially fruitful in the study of finite groups where it allows one to decompose the group algebra of a finite group over an algebraically closed field as a finite product of matrix rings, assuming that the characteristic of the field does not divide the order of the group.  In general, however, it is a difficult problem to classify the irreducible representations of a ring and so one often settles for a coarser understanding of the representation theory of an algebra by instead classifying its primitive ideals; that is, the ideals that annihilate a simple left module of the ring.

In general, the set of annihilators of simple left modules and the set of annihilators of simple right modules do not coincide. Bergman \cite{Ber}, while still an undergraduate, gave an example of a ring in which $(0)$ is a right primitive but not a left primitive ideal; in practice, however, the two sets often agree and so we will not concern ourselves with such pathological examples.  

The primitive ideals of a ring are necessarily prime and thus ${\rm Prim}(R)$, the set of primitive ideals of $R$, forms a distinguished subset of ${\rm Spec}(R)$.  Our objective then becomes to identify the prime ideals in ${\rm Spec}(R)$ that are primitive.  
A fundamental result in this direction is the pioneering work of Dixmier and Moeglin \cite{Dix77, Moe80}, who together showed that if $L$ is a finite-dimensional complex Lie algebra then the primitive ideals of the enveloping algebra $U(L)$ are just the prime ideals of ${\rm Spec}(U(L))$ that are locally closed in the Zariski topology; in addition, they showed that the primitive prime ideals are exactly those that are \emph{rational}; that is, primes $P$ such that the Goldie ring of quotients of $U(L)/P$ has centre exactly the base field $\mathbb{C}$.  In general, for a noetherian algebra $R$ over a field $k$, a prime ideal $P$ of $R$ is rational if ${\rm Frac}(R/P)$ has centre that is an algebraic extension of the base field $k$, where for a prime Goldie ring $S$ we let ${\rm Frac}(S)$ denote its Goldie ring of quotients, which is a simple Artinian ring that plays the role that the field of fractions plays for commutative domains.

Thus Dixmier and Moeglin's fundamental result says that for $P\in {\rm Spec}(U(L))$ we have the following equivalences:
$$P ~{\bf locally~closed}\iff P~{\bf primitive}\iff P~{\bf rational}.$$
 
These equivalences have since been shown to hold for other classes of rings, including (and this list is by no means exhaustive) affine algebras satisfying a polynomial identity (PI algebras, for short) \cite[2.6]{Von1}, group algebras of nilpotent-by-finite groups \cite{Z}, many quantum algebras \cite{GoLet}, \cite[II.8.5]{BrGo}), affine cocommutative Hopf algebras of finite Gelfand-Kirillov dimension in characteristic zero \cite{BL14}, Hopf Ore extensions of affine commutative Hopf algebras \cite{BSM18}, noetherian twisted homogeneous coordinate rings of complex curves and surfaces \cite{BRS10}, Hopf algebras of Gelfand-Kirillov dimension two (under mild homological assumptions) \cite{GZ}, finitely generated complex connected graded domains of Gelfand-Kirillov dimension two that are generated in degree one (cf. Artin and Stafford \cite{AS} and \cite{BRS10}).  

An important remark, which we make use of throughout the paper, is that for algebras that satisfy the Nullstellensatz, to prove that the above equivalences hold, it suffices to prove that rational prime ideals are necessarily locally closed (see, for example, \cite[Theorem 2.3]{Importance}).  For background on the Nullstellensatz, we refer the reader to \cite[II.7]{BrGo}. In particular, the algebras we consider in this paper all satisfy the Nullstellensatz by a result of Irving \cite[Theorem 2]{Irving} and so we will make use of the following facts.
\begin{fact} \label{fact2}
Let $R$ be a commutative Noetherian Jacobson ring.  Then an Ore extension $R[x;\sigma,\delta]$ satisfies the Nullstellensatz. 
\end{fact}
\begin{fact} \label{fact}
Let $T$ be a noetherian algebra that satisfies the Nullstellensatz.  Then if
\begin{equation}
P~{\rm rational}\implies P ~{\rm locally~closed}\end{equation} for all $P\in {\rm Spec}(T)$ then $T$ satisfies the Dixmier-Moeglin equivalence.
\end{fact}

In honour of the work of Dixmier and Moeglin, we now say that a ring $R$ satisfies the \emph{Dixmier-Moeglin equivalence} if a prime ideal $P$ is primitive if and only if it is rational, if and only if $\{P\}$ is locally closed in ${\rm Spec}(R)$.

Lorenz \cite{Lor} shortly thereafter gave an example of a polycyclic group that is not nilpotent-by finite whose group algebra does not satisfy the Dixmier-Moeglin equivalence.  Concretely, his example is an algebra of the form $R[z^{\pm 1};\sigma]$ where $R$ is the ring of Laurent polynomials over $\mathbb{C}$ in two variables.  In his example, $(0)$ is a primitive ideal of the group algebra that is not locally closed.  Later in \cite{BLLM17}, it was shown that for each integer $d\ge 4$ there is a finitely generated commutative integral domain $R$ over the complex numbers with a $\mathbb{C}$-linear derivation such that $R[x;\delta]$ is a domain of Gelfand-Kirillov dimension $d$, which does not satisfy the Dixmier-Moeglin equivalence.  In light of these examples, it is of special significance to ask when an Ore extension of a commutative ring satisfies the Dixmier-Moeglin equivalence.  We show that in some sense the counterexamples of \cite{BLLM17} are optimal, by showing that for Gelfand-Kirillov dimension $< 4$ the Dixmier-Moeglin equivalence holds.  For background related to growth and Gelfand-Kirillov dimension, we refer the interested reader to the book of Krause and Lenagan \cite{KL}.

 \begin{theorem} \label{thm:main} Let $k$ be a field of characteristic zero, let $R$ be a finitely generated commutative $k$-algebra that is a domain, let $\sigma$ be a $k$-algebra automorphism of $R$, and let $\delta$ be a $k$-linear $\sigma$-derivation of $R$.  Then if $T=R[x;\sigma,\delta]$ has Gelfand-Kirillov dimension strictly less than four then $T$ satisfies the Dixmier-Moeglin equivalence.
\end{theorem}
In the case that we have a skew extension of derivation type $T=R[x;\delta]$ (that is, an Ore extension in which the automorphism $\sigma$ is the identity), we can show a stronger result: that if $R$ is a finitely generated prime noetherian algebra of Gelfand-Kirillov dimension $<3$ that satisfies both the Nullstellensatz and the Dixmier-Moeglin equivalence, then $T$ does too (see Theorem \ref{thm:Leon}).  

The outline of this paper is as follows.  In \S2, we consider Ore extensions of Gelfand-Kirillov dimension two; in \S3 we look at Ore extensions of derivation type and prove Theorem \ref{thm:Leon}; finally, in \S4 we prove Theorem \ref{thm:main}.

We make the remark that throughout this paper, we will study rings of the form $R[x;\sigma,\delta]$ and we will often consider related rings of the form $S[x;\sigma',\delta']$ where $S$ is either a homomorphic image of, a subring of, or a localization of $R$, and the maps $\sigma'$ and $\delta'$ are induced from $\sigma$ and $\delta$ in a canonical way.  In this case, as an abuse of notation, we will write $S[x;\sigma,\delta]$ with the understanding that, strictly speaking, the maps $\sigma$ and $\delta$ are not technically maps from $S$ to itself, but can be obtained from the original maps $\sigma$ and $\delta$ on $R$ in a natural way.



\section{Ore extensions of commutative algebras of Gelfand-Kirillov dimension one}
Given a commutative noetherian ring $R$ and an Ore extension $T:=R[x;\sigma,\delta]$ of $R$,
Goodearl \cite[Theorem 3.1]{Good} showed that for $P\in {\rm Spec}(T)$ with $I:=P\cap R$, we have that the ideal $P$ falls into one of the following three cases:
\begin{enumerate}
\item[(a)] $T/P$ is commutative;
\item[(b)] $T/P$ is not commutative and $I$ is a $(\sigma,\delta)$-prime ideal that is $\sigma$-prime;
\item[(c)] $T/P$ is not commutative and $I$ is a $(\sigma,\delta)$-prime ideal that is $\delta$-prime but not $\sigma$-prime.
\end{enumerate}
Throughout this paper we will refer to the primes of ${\rm Spec}(T)$ as being of types (a)--(c), making reference to the above trichotomy of Goodearl.  We note that in case (c), Goodearl shows in addition that $R/I$ has a unique associated prime, which contains $(1-\sigma)(R)$.   
\begin{definition} Let $k$ be a field and let $R$ be a finitely generated $k$-algebra.  We say that a subspace $V$ of $R$ is a \emph{frame} of $R$ if $1\in V$, $V$ is finite-dimensional, and $V$ generates $R$ as a $k$-algebra.
\end{definition}
We begin by stating an important result of Zhang.

\begin{proposition} (\cite[Lemma 4.1]{Zhang}) Let $R$ be a finitely generated $k$-algebra and suppose that $\sigma$ is a $k$-algebra automorphism of $R$ that preserves a frame $V$ and that $\delta$ is a $\sigma$-derivation.  Then if $W=V+kx\subseteq T:=R[x;\sigma,\delta]$ then
$W$ is a frame of $T$ and there is a positive integer $L>0$ such that 
${\rm dim}(W^n) \le n\cdot {\rm dim}(V^{Ln})$ for all $n\ge 1$.  In particular,
${\rm GKdim}(R[x;\sigma,\delta])={\rm GKdim}(R)+1$.
\label{prop:Zhang}
\end{proposition}
\begin{proof} Zhang's proof shows that if we choose $L$ so that $\delta(V)\subseteq V^L$, then the above inequality holds. 
\end{proof}

\begin{lemma} Let $k$ be an algebraically closed field and let $R$ be a finitely generated commutative $k$-algebra of Gelfand-Kirillov dimension one that is a domain.  If $\sigma$ is a $k$-algebra automorphism of $R$ then $\sigma$ preserves a frame.
\label{lem:frame}
\end{lemma}
\begin{proof} Let $\overline{R}$ denote the integral closure of $R$ in ${\rm Frac}(R)$.  Then $\sigma$ lifts to an automorphism of $\overline{R}$.  Let $X={\rm Spec}(\overline{R})$, which is a smooth irreducible affine curve over $k$, and let $\tau$ denote the automorphism of $X$ corresponding to $\sigma$. 
Then $X$ can be embedded as an open subset of a projective variety $Y$, and replacing $Y$ by
its normalization if necessary (doing this does not affect the open normal subset $X$), we may assume that $Y$ is a smooth projective curve and $\tau$ induces a birational map of $Y$, which is in fact an automorphism since $Y$ is a smooth curve.

If $Y$ has genus $\ge 2$ then $\tau$ has finite order on $Y$ and so $\sigma$ is a finite-order automorphism of $R$ and so $\sigma$ preserves a frame.  If $Y$ has genus $1$ then $Y$ is an elliptic curve and since $Y\setminus X$ is finite and preserved by $\tau$, some iterate of $\tau$ has a fixed point.  Then after a change of coordinates, we may assume this fixed point is the identity of $Y$ and so $\tau$ is a group automorphism of $Y$ and thus again has finite order and so once again $\sigma$ preserves a frame.  

Hence we may assume that $Y$ is of genus zero and thus $Y\cong \mathbb{P}^1$.  Then if $$\#(Y\setminus X)\ge 3$$ then some iterate of $\tau$ has at least three fixed points and thus is the identity, so once again $\sigma$ has finite order and we are done.  Thus we may assume that $\#(Y\setminus X)\in \{0,1,2\}$ and that $Y=\mathbb{P}^1$, and since $X$ is affine and $Y$ is projective, $Y\setminus X$ is non-empty.  Then after a change of coordinates we may assume that $Y\setminus X$ is either $\{\infty\}$ or $\{0,\infty\}$ and so $X=\mathbb{A}^1$ or $X=\mathbb{A}^1\setminus \{0\}$ and so $\sigma$ is either linear or affine linear and so the induced map on $\overline{R}$ preserves a frame.  Then if $V$ is a frame of $R$ then $V$ is contained is some $\sigma$-invariant frame $W$ of $\overline{R}$ and so $W\cap R$ is a frame of $R$ that is invariant under $\sigma$.  The result follows.
\end{proof}
\begin{lemma} Let $k$ be an algebraically closed field and let $R$ be a finitely generated reduced commutative $k$-algebra of Gelfand-Kirillov dimension at most one and let $\sigma$ be a $k$-algebra automorphism of $R$.  Then $\sigma$ preserves a frame of $R$.  In particular every one-step Ore extension of $R$ has Gelfand-Kirillov dimension two. 
\label{lem:reduced}
\end{lemma}
\begin{proof} If $R$ is of Gelfand-Kirillov dimension zero then $R$ is finite-dimensional and so $R$ is a frame. Thus we may assume $R$ has Gelfand-Kirillov dimension one.  Let $P_1,\ldots ,P_s$ denote the minimal primes of $R$.  Then $\sigma $ permutes these primes and hence there is some $m\ge 1$ such that $\sigma^m(P_i)=P_i$ for $i=1,\ldots ,s$.  If $\sigma^m$ preserves a frame $W$ then $\sigma$ preserves the frame $\sum_{i=0}^{m-1} \sigma^i(W)$ and thus it suffices to show that $\sigma^m$ preserves a frame.  Now $R/P_i$ is a domain of Gelfand-Kirillov dimension one and so $\sigma^m$ preserves a frame $W_i$ of $R/P_i$ for $i=1,\ldots, s$ by Lemma \ref{lem:frame}.  Since the intersection of the $P_i$ is $(0)$, $R$ embeds in $\prod_{i=1}^s R/P_i$ via the map $\pi$ which sends $r$ to $(r+P_1,\ldots ,r+P_s)$.  We let $V$ be a frame for $R$.  Then there is some $N$ such that $V+P_i\subseteq W_i^N$ for $i=1,\ldots ,s$.  We let 
$W=\{r\in R\colon \pi(r)\in W_1^N\times \cdots \times W_s^N\}.$  Then since $\pi$ is injective, $W$ is finite-dimensional and by construction $W$ contains $V$ and so $W$ is a frame for $R$.  Note that if $r\in W$ then $\pi(\sigma^m(r)) = (\sigma^m(r)+P_1,\ldots ,\sigma^m(r)+P_s)$.  Since $W_1,\ldots ,W_s$ are $\sigma^m$-invariant frames of $R/P_1,\ldots ,R/P_s$ respectively,
we see that if $r\in W$ then $\sigma^m(r)\in W$, and so $W$ is a $\sigma^m$-invariant frame of $R$.  The ``in particular'' clause now follows from Proposition \ref{prop:Zhang}.
\end{proof}
\begin{remark}
In general, a $k$-algebra automorphism of a finitely generated commutative $k$-algebra of Gelfand-Kirillov dimension one need not preserve a frame.  For example if $R=k[x,y^{\pm 1}]/(x^2)$ and $\sigma$ is the $k$-algebra automorphism of $R$ given by $\sigma(x)=xy$ and $\sigma(y)=y$ then $\sigma$ does not preserve a frame of $R$ since $\{\sigma^i(x)\colon i\ge 0\} = \{xy^i \colon i\ge 0\}$ which spans an infinite-dimensional subspace of $R$, and so $x$ cannot be in a $\sigma$-invariant frame of $R$.
\end{remark}

\begin{proposition} Let $k$ be an algebraically closed field, let $R$ be a finitely generated reduced commutative $k$-algebra, let $\sigma$ be a $k$-algebra automorphism, and let $\delta$ be a $k$-linear $\sigma$-derivation of $R$.  If $T:=R[x;\sigma,\delta]$ has Gelfand-Kirillov dimension $<3$ and $R$ is either reduced or $\sigma$ preserves a frame then $T$ satisfies the Dixmier-Moeglin equivalence.
\label{prop:GK2}
\end{proposition}
\begin{proof} 
Since $3> {\rm GKdim}(T)\ge {\rm GKdim}(R)+1$, we see $R$ is either of Gelfand-Kirillov dimension zero or one.  Hence by Lemma \ref{lem:reduced}, $\sigma$ preserves a frame if $R$ is reduced. Thus we may assume that $\sigma$ preserves a frame. By Fact \ref{fact}, it suffices to show that rational primes of $T$ are locally closed.  

Let $P$ be a prime ideal of $T$.  Then if $T/P$ is commutative then $P$ is locally closed if and only if $P$ is rational and thus it suffices to consider prime ideals in ${\rm Spec}(T)$ of types (b) and (c).  In these cases we can let $I=R\cap P$, which is $(\sigma,\delta)$-prime, and replace $R$ by $R/I$ and $T$ by $(R/I)[x;\sigma,\delta]$.  Then $T/P$ is a prime homomorphic image of $(R/I)[x;\sigma,\delta]$, which is now a prime ring \cite[Prop. 3.3]{Good}.  Thus we may replace $R$ with $R/I$ and $T$ with $(R/I)[x;\sigma,\delta]$ and we may assume without loss of generality that $T$ is a prime ring.
If after these modifications $R$ is of Gelfand-Kirillov dimension zero then $R$ is finite-dimensional and so $T$ is of Gelfand-Kirillov dimension $1$ and thus satisfies a polynomial identity by the Small-Stafford-Warfield theorem \cite{SSW}.  It follows that $T$ satisfies the Dixmier-Moeglin equivalence in this case.  Thus we may assume that ${\rm GKdim}(R)=1$ and ${\rm GKdim}(T)=2$.  Then since $\sigma$ preserves a frame of $R$ and since $R$ has linear growth we also have that $T$ is of quadratic growth by Proposition \ref{prop:Zhang}.

Let $P$ be a rational prime ideal of $T$.  We must show that $P$ is locally closed.
Then if $P$ is of type (a) then, as before, $P$ is rational if and only if $P$ is maximal and 
hence locally closed.  Thus we may assume that $P$ is of type (b) or (c) and so letting 
$I=R\cap P$ we have $I$ is a $(\sigma,\delta)$-ideal of $R$.  Thus we may replace $R$ by 
$R/I$ and $T$ by $(R/I)[x;\sigma,\delta]$ and assume that $P\cap R=(0)$.  Then if $T$ has 
Gelfand-Kirillov dimension $\le 1$ or if $(0)$ is not rational then $T$ is PI \cite{BS} and hence satisfies the 
Dixmier-Moeglin equivalence.  Thus we may assume that $T$ is of quadratic growth and 
that $(0)$ is a rational prime ideal.  Since $T/P$ is of Gelfand-Kirillov dimension one for $P$ nonzero, it satisfies a polynomial identity and hence satisfies the Dixmier-Moeglin equivalence.  So it suffices to show that $(0)$ is locally closed to obtain the Dixmier-Moeglin equivalence in this case.

We pick a nonzero commutator $\theta\in T$.  Then $T$ has finitely many primes $Q_1,\ldots,  Q_s$ of co-GK one \cite{Bell}.  We claim that every nonzero prime ideal $Q$ of $T$ either contains some $Q_i$ or it contains $\theta$.  To see this, note that if $Q$ is of type (a) then it contains $\theta$ and so we may assume that $Q$ is of type (b) or (c).  We let $J=Q\cap R$ and let $S=T/JT$, which is a prime algebra of Gelfand-Kirillov dimension one.  Then $J$ is nonzero by the above remarks and $T/Q$ is a homomorphic image of $T/JT$, and so $Q\supseteq JT$.  But since $JT$ is a prime ideal of $T$ of co-GK $1$ we have $JT=Q_i$ for some $i$ and so the result follows.  Thus we see that the intersection of the nonzero prime ideals of $T$ is nonzero and so $(0)$ is locally closed.  The result follows.
\end{proof}
\section{Ore extensions of derivation type}
In this section we consider Ore extensions of the form $T:=R[x;\delta]$ where $R$ is a finitely generated noetherian algebra of Gelfand-Kirillov dimension $<3$ that satisfies both the Nullstellensatz and the Dixmier-Moeglin equivalence.  We show under these conditions that
the Ore extension $T$ also satisfies the Dixmier-Moeglin equivalence as long at $T$ satisfies the Dixmier-Moeglin equivalence.  The assumption that $R$ be of Gelfand-Kirillov dimension $< 3$ is necessary as there are examples of affine commutative algebras $R$ of Gelfand-Kirillov dimension three such that $R[x;\delta]$ does not satisfy the Dixmier-Moeglin equivalence \cite{BLLM17}.
On the other hand, if $\delta$ preserves a frame $V$ of $R$ and, in addition, $R$ is a $\mathbb{C}$-algebra then $R[x;\delta]$ satisfies the Dixmier-Moeglin equivalence \cite[Theorem 1.6]{BWW}.  We first show there is a connection between having bounded matrix images and satisfying the Dixmier-Moeglin equivalence for noetherian algebras of Gelfand-Kirillov dimension two.  We recall that if $k$ is an algebraically closed field and $A$ is a $k$-algebra then $A$ has \emph{bounded matrix images} if the set of $n$ for which there exists a surjective $k$-algebra homomorphism from $A$ to $M_n(k)$ is finite.  More generally, if $k$ is not algebraically closed, we say that $A$ has bounded matrix images if $A\otimes_k \bar{k}$ has bounded matrix images.  Small has asked whether a prime noetherian algebra of quadratic growth necessarily has bounded matrix images (cf. \cite[Question 3.1]{Baffine}), and there are currently no known examples of such an algebra with unbounded matrix images.

\begin{proposition}
Let $k$ be a field and let $R$ be a prime noetherian $k$-algebra of Gelfand-Kirillov dimension less than three that satisfies the Nullstellensatz.  Then $R$ satisfies the Dixmier-Moeglin equivalence if and only if $R$ has bounded matrix images.
\end{proposition}
\begin{proof} By Fact \ref{fact}, to show that $R$ satisfies the Dixmier-Moeglin equivalence, it suffices to show that rational primes of $R$ are locally closed.  

Suppose that $R$ has unbounded matrix images.  Then let $\mathcal{S}$ denote the set of prime ideals of finite codimension.  Then $$I:=\bigcap_{Q\in \mathcal{S}} Q$$ has the property that $R/I$ has unbounded matrix images and thus $R/I$ must have Gelfand-Kirillov dimension at least two, since otherwise it would have Gelfand-Kirillov dimension $\le 1$ and hence would be PI \cite{SSW}, which contradicts the fact that it has unbounded matrix images.  It follows that $I=(0)$ since ${\rm GKdim}(R/J)\le {\rm GKdim}(R)-1$ for a nonzero ideal $J$ of a prime noetherian algebra.  It follows that $(0)$ is not locally closed in ${\rm Spec}(R)$.  If $R$ satisfies the Dixmier-Moeglin equivalence then $(0)$ cannot be rational and so $R$ is PI \cite{BS}.  Thus we see that if $R$ has unbounded matrix images then $R$ does not satisfy the Dixmier-Moeglin equivalence.  

Now suppose that $R$ has bounded matrix images.  If $R$ is PI then it satisfies the Dixmier-Moeglin equivalence \cite{Von1} and so we may assume that $R$ is not PI and so $(0)$ is a rational prime ideal \cite{BS}.  We now claim that $(0)$ is locally closed in ${\rm Spec}(R)$.  To see this, notice that if $P$ is a nonzero prime ideal of $R$ then $R/P$ is of Gelfand-Kirillov dimension $\le 1$ and hence is PI.  Since $R$ has bounded matrix images, there is some $m$ such that $R/P$ satisfies the $m$-th standard identity.  In particular, since $R$ does not satisfy this identity, there is some nonzero element of $R$ that is in every nonzero prime ideal of $R$ and so $(0)$ is locally closed.  Next notice that if $P$ is a nonzero rational ideal then $R/P$ has Gelfand-Kirillov dimension $\le 1$ and thus is PI.  Since $P$ is rational it follows that $P$ is maximal and thus locally closed.  It follows that $R$ satisfies the Dixmier-Moeglin equivalence if $R$ has bounded matrix images.  The result follows.
\end{proof}

We now consider rings of the form $R[x;\delta]$ when $R$ is PI. We first give a preliminary result. 
\begin{lemma} \label{lem:Z}
Let $k$ be a field and let $R$ be a finitely generated commutative $k$-algebra of Gelfand-Kirillov dimension $\le 2$ that is an integral domain.  Then if $\delta$ is a $k$-linear derivation of $R$ such that $(0)$ is a rational prime of $R[x;\delta]$ then the intersection of the nonzero $\delta$-invariant prime ideals of $R$ is nonzero.
\end{lemma}
\begin{proof} If $k$ is of characteristic $p>0$, then $\delta(a^p)=0$ for every $a\in R$ and so $R[x;\delta]$ is a finite module over its centre, and hence $R[x;\delta]$ is PI and so it satisfies the Dixmier-Moeglin equivalence.  Thus we may assume without loss of generality that $k$ has characteristic zero.  By Fact \ref{fact} it suffices to show that rational prime ideals are locally closed.

Every nonzero prime ideal of $R$ is either height one or maximal (or both if $R$ has Gelfand-Kirillov dimension $1$).   Since $(0)$ is rational, a result of Jouanolou \cite{Jou} (see also  \cite[Theorem 6.1]{BLLM17} for showing the connection) gives that there are only finitely many height one $\delta$-invariant prime ideals of $R$.  Hence it suffices to show that the intersection of the $\delta$-invariant maximal ideals of $R$ is nonzero.  Since $(0)$ is a rational prime ideal of $R[x;\delta]$, there is some $y\in R$ such that $\delta(y)\neq 0$.  Let $P$ be a $\delta$-invariant maximal ideal of $R$.  Then $R/P$ is a finite extension of $k$ and hence there is an irreducible polynomial $p(t)\in k[t]$ such that $p(y)\in P$.  Then since $P$ is $\delta$-invariant, we see $\delta(p(y))=p'(y)\delta(y)\in P$.  Then $p'(t)$ and $p(t)$ have gcd one and so $\delta(y)\in P$. Thus the intersection of the $\delta$-invariant maximal ideals of $R$ contains $\delta(y)$ and so the result follows.
\end{proof}
\begin{lemma} \label{lem:PIcase1}
    Let $R$ be a finitely generated prime PI algebra of Gelfand-Kirillov dimension $< 3$.  Then if $(0)$ is a rational prime ideal of $T:=R[x;\delta]$ then $(0)$ is locally closed. \end{lemma}
\begin{proof}
A finitely generated prime PI algebra has integer Gelfand-Kirillov dimension, so in fract $R$ has Gelfand-Kirillov dimension at most two.
Since $(0)$ is rational, we have ${\rm Frac}(R)[x;\delta]$ is simple and so every nonzero prime ideal $P$ of $T$ intersects $R$ non-trivially and $I:=R\cap P$ is a $\delta$-invariant prime ideal of $R$.  Thus it suffices to show that the intersection of the nonzero $\delta$-invariant prime ideals of $R$ is nonzero.

    Let $Z := Z(R)$ denote the centre of $R$. Then $Z$ is a domain and there exists a nonzero $z\in Z$ such that $R[1/z]$ is a free $Z[1/z]$-module of finite rank \cite[Prop. 4.4(3)]{ASZ}.

Since $R[1/z]$ is a finitely generated algebra, $Z[1/z]$ is also finitely generated by the Artin-Tate lemma (see \cite[Proposition 2]{SM}).  Then $\delta$ induces a derivation of $R[1/z]$, which then restricts to a derivation of $Z[1/z]$.  Then since $(0)$ is a rational prime ideal of $R[1/z][x;\delta]$, $(0)$ is also a rational prime ideal of $Z[1/z][x;\delta]$ \cite[Corollary 1.2]{Letzter}.

Now by Lemma \ref{lem:Z} we have that the intersection of the nonzero $\delta$-invariant prime ideals of $Z[1/z]$ is nonzero and hence contains a nonzero element $u$.  We now claim that $u$ is in every nonzero $\delta$-invariant prime ideal of $R[1/z]$.  To see this, let $P$ be a nonzero $\delta$-invariant prime ideal of $R[1/z]$.  Then $I:=P\cap Z[1/z]$ is nonzero by Posner's theorem and is evidently $\delta$-invariant. Then let $Q_1,\ldots ,Q_s$ be the minimal prime ideals above $I$.  Then by \cite[14.2.3]{McR}, $Q_1,\ldots ,Q_s$ are all $\delta$-invariant and hence $u\in \sqrt{I}=Q_1\cap\cdots \cap Q_s$.  It follows that $u^n\in I$ for some $n$
and so $u^n\in P$.  But now $u$ is central in $R$ and since $I$ is prime we then have $u\in I$ and hence $u\in P$.  Thus we obtain the claim. 

Finally, note that every nonzero $\delta$-invariant prime ideal of $R$ either contains $z$, or it survives in the localization $R[1/z]$, and hence contains $u$, and so the intersection of the nonzero $\delta$-invariant prime ideals of $R$ is nonzero, since $R$ is itself a prime ring.
\end{proof}


\begin{proposition}\label{CasenonPI}
    Let $k$ be a field, let $R$ be a prime noetherian $k$-algebra of Gelfand-Kirillov dimension less than three, and let $\delta$ be a $k$-linear derivation of $R$.  If $R$ satisfies both the Nullstellensatz and the Dixmier-Moeglin equivalence and if $R$ is not PI then $T:=R[x;\delta]$ satisfies the Dixmier-Moeglin equivalence whenever $T$ satisfies the Nullstellensatz.
    \end{proposition}
\begin{proof}
We first consider the case when $(0)$ is a rational prime ideal of $T$.  In this case ${\rm Frac}(R)[x;\delta]$ is simple and hence every nonzero ideal of $T$ intersects $R$ non-trivially.  In particular, if $P$ is a nonzero prime ideal of $T$ then $P\cap R$ is a nonzero $\delta$-invariant prime ideal of $R$ and so the intersection of the nonzero prime ideals of $T$ contains the intersection of the nonzero prime ideals of $R$.

Since $R$ is not PI, $(0)$ is a rational prime ideal of $R$ and so the intersection of the nonzero prime ideals of $R$ is nonzero, since $R$ satisfies the Dixmier-Moeglin equivalence.  Thus $(0)$ is locally closed in ${\rm Spec}(T)$.  So we have shown that if $(0)$ is rational then $(0)$ is locally closed.  Now suppose that $P$ is a nonzero rational prime ideal of $T$ and let $I=P\cap R$.  As noted above, $I\neq (0)$ and $I$ is a $\delta$-invariant prime ideal of $R$.  Then $R/I$ is a prime ring of Gelfand-Kirillov dimension $\le 1$ and $T/P$ is a homomorphic image of $(R/I)[x;\delta]$.  By the Small-Warfield theorem \cite{SW}, $R/I$ is a finite module over its centre, which is invariant under $\delta$.
Then by Lemma \ref{lem:PIcase1} we see that $(0)$ is locally closed in $T/P$.  Hence we have shown that rational prime ideals of $T$ are locally closed when $(0)$ is rational prime ideal of $T$.  

Next we consider the case when $(0)$ is not a rational prime ideal of $T$ and let $P$ be a rational prime ideal of $T$.  If $I:=P\cap R\neq (0)$, then as in the preceding case we have that $T/P$ is a homomorphic image of $(R/I)[x;\delta]$ and since $R/I$ has Gelfand-Kirillov dimension $\le 1$, we see that $(R/I)[x;\delta]$ satisfies the Dixmier-Moeglin equivalence as before by reducing to Proposition \ref{prop:GK2}.  Thus we may assume that $P\cap R=(0)$.  Then $P$ survives in the localization 
${\rm Frac}(R)[x;\delta]$ and generates a nonzero prime ideal $\widetilde{P}$.  Then since ${\rm Frac}(R)$ is simple, $\widetilde{P}$ contains a nonzero monic polynomial in $x$ and so 
${\rm Frac}(R)[x;\delta]/\widetilde{P}$ is a finite ${\rm Frac}(R)$-module and thus is artinian; since it is prime, it is simple.  It follows that every prime ideal $Q$ that strictly contains $P$ must intersect $R$ non-trivially and as noted before, $Q\cap R$ is then a nonzero prime ideal of $R$.  Since $(0)$ is a rational prime ideal of $R$ and $R$ satisfies the Dixmier-Moeglin equivalence, we see that the intersection of the nonzero ideals of $R$ is nonzero and so the intersection of the prime ideals of $T$ that strictly contain $P$ is nonzero.  The result follows.
\end{proof}

   \begin{theorem}
    Let $k$ be a field, let $R$ be a finitely generated prime noetherian $k$-algebra of Gelfand-Kirillov dimension less than three that satisfies the Nullstellensatz, and let $\delta$ be a $k$-linear derivation of $R$. If $R$ satisfies the Dixmier-Moeglin equivalence and if $T = R[x,\delta]$ satisfies the Nullstellensatz then $T$ also satisfies the Dixmier-Moeglin equivalence.
    \label{thm:Leon}
\end{theorem}
\begin{proof}
If $R$ is not PI this follows from Proposition \ref{CasenonPI}.
Thus we may assume that $R$ is PI and it suffices to show that a rational prime ideal of $T$ is locally closed.  If $(0)$ is a rational prime ideal of $T$ then $(0)$ is locally closed by Lemma \ref{lem:PIcase1}.  Thus it suffices to show that every nonzero rational prime ideal $P$ is locally closed.  Let $I:=P\cap R$.  If $I=(0)$, then $P$ survives in the localization ${\rm Frac}(R)[x;\delta]$, and since $P$ is nonzero, the corresponding ideal, $\tilde{P}$, in ${\rm Frac}(R)[x;\delta]$ contains a monic polynomial in $x$ and so ${\rm Frac}(R)[x;\delta]/\tilde{P}$ is a finite module over ${\rm Frac}(R)$ and hence is PI, which contradicts the fact that $P$ is rational unless $R$ is simple and hence of Gelfand-Kirillov dimension zero.  In this case $T$ is of Gelfand-Kirillov dimension one and hence PI and so $T$ satisfies the Dixmier-Moeglin equivalence in this case.  On the other hand, if $I$ is nonzero then
$T/P$ is a homomorphic image of $S:=(R/I)[x;\delta]$, and $R/I$ is a finitely generated prime algebra of Gelfand-Kirillov dimension $\le 1$ and hence is PI. If $(0)$ is a rational prime ideal of $S$ then $(0)$ is locally closed by Lemma \ref{lem:PIcase1}.  If $(0)$ is not rational then since $S$ is of Gelfand-Kirillov dimension two, it is PI by \cite{BS}.  But $P$ corresponds to a rational prime ideal $P'$ of $S$ with $P'\cap S=(0)$, and so this is impossible since $S$ is PI. The result follows.
\end{proof}

\section{Proof of Theorem \ref{thm:main}}
In this section we give the proof of Theorem \ref{thm:main}.  We first need two preliminary lemmas.

\begin{lemma}
\label{lem:R0}
Let $k$ be an algebraically closed field, let $R$ be a commutative $k$-algebra with a frame-preserving $k$-algebra automorphism $\sigma$ and $\sigma$-derivation $\delta$ and let $T=R[x;\sigma,\delta]$.  If $\sigma$ is not the identity map and $\delta$ is not inner and $R$ is $\sigma$-prime then $R_0:=\{r\in R\colon \sigma(r)=r\}$ is central in $T$.
\end{lemma}
\begin{proof}
Suppose that $\sigma$ is not the identity map.  Then since $\sigma$ preserves a frame, either there is some nonzero $f\in R$ with $\sigma(f)=\lambda f$ and $\lambda\neq 1$ or there is some $f\in R$ with $\sigma(f)-f=1$.  In the latter case, $\delta$ is inner, so we may assume there is some nonzero $f$ with $\sigma(f)=\lambda f$ with $\lambda\neq 1$.  Now by \cite[Lemma 4.14]{BSM18} we have $f$ is regular.  Now let $r\in R_0$.  Then expanding both sides of $\delta(rf)=\delta(fr)$ and since $\sigma(r)=r$ and $\sigma(f)=\lambda f$ we then see that $(\sigma(f)-f)\delta(r) = 0$ and so $\delta(r)=0$.  Now this means every element of $R_0$ commutes with $x$ and so $R_0$ is central in $T$, as claimed.  
\end{proof}
\begin{lemma}
Let $k$ be an algebraically closed field, let $R$ be a finitely generated commutative $k$-algebra, let $\sigma$ be a frame-preserving $k$-algebra automorphism of $R$, and let $\delta$ be a $\sigma$-derivation of $R$.  Suppose that $R$ is $\sigma$-prime, $\sigma$ is not the identity map, $\delta$ is not inner, $(0)$ is a rational prime of $T$, and suppose that there exists a prime ideal of $T$ of type (c).  Then there is a nonzero $\delta$-prime ideal $I$ of $R$ such that $P\cap R=I$ for every prime ideal $P$ of type (c). \label{lem:typec} 
\end{lemma}
\begin{proof}
Let $P$ be a rational prime ideal of $T$ of type (c). 
Goodearl shows there is a unique associated prime ideal $Q=Q(P)$ of $R/I$ and $\sigma(r)-r\in Q$ for all $r\in R$.  
Let $W$ be a $\sigma$-invariant frame of $R$.  Then since $(0)$ is rational and $R$ is $\sigma$-prime, by Lemma \ref{lem:R0} we have that $R_0=k$, where $R_0=\{r\colon \sigma(r)=r\}$.  It follows that $(\sigma-{\rm id})|_W$ has one-dimensional kernel and so the image of $(\sigma-{\rm id})|_W$ has codimension one in $W$.  Thus $W/Q\cap W$ is one-dimensional.  Since this is true for every $\sigma$-invariant frame and since $R=\bigcup W^n$, we then see that $R/Q$ is a $1$-dimensional $k$-vector space.   Furthermore this argument shows that 
$$Q={\rm Span}_k\{r-\sigma(r)\colon r\in R\}.$$
Hence we have shown that if $P$ and $P'$ are two primes of type (c) then 
$$Q(P)=Q(P')={\rm Span}_k\{r-\sigma(r)\colon r\in R\}.$$  
Now suppose that we have two distinct primes $P$ and $P'$ of type (c).  Then let $Q=Q(P)=Q(P')$ and let $I=R\cap P$ and $I'=R\cap P'$.  Goodearl \cite[Proposition 3.11]{Good} shows that $I$ and $I'$ are the largest $\delta$-prime ideals contained in $Q$ and so $I=I'$.
The result follows.
\end{proof}

 \begin{proposition} \label{prop:main} Let $k$ be an algebraically closed field, let $R$ be a finitely generated reduced commutative $k$-algebra of Gelfand-Kirillov dimension two, let $\sigma$ be a $k$-algebra automorphism of $R$ that preserves a frame, and let $\delta$ be a $k$-linear $\sigma$-derivation of $R$.  Then $T:=R[x;\sigma,\delta]$ satisfies the Dixmier-Moeglin equivalence.
\end{proposition}
\begin{proof}
If $\sigma$ is the identity then the result follows by Theorem \ref{thm:Leon}, since every prime ideal of $T$ intersects $R$ in a prime ideal, and so we may assume that $R$ is prime in this case; if $\delta$ is inner, then $R[x;\sigma,\delta]\cong R[x;\sigma]$ and the result follows by \cite[Theorem 1.6]{BWW}.\footnote{Although the theorem in \cite{BWW} has a hypothesis that the base field be uncountable, this is only used to ensure that $R[x;\sigma]$ satisfies the Nullstellensatz.  But since $R$ is commutative, we can use Fact \ref{fact2} in this case and then follow the original proof.}
Thus for the remainder of the proof, we assume that $\sigma$ is not the identity and $\delta$ is not inner.

We now divide the proof into two cases.
\vskip 2mm
{\bf Case I.} $P$ is a rational prime ideal of ${\rm Spec}(T)$ with $P\cap R$ nonzero.
\vskip 2mm
Then either $P$ is of type (a) in which case $T/P$ is commutative and hence $P$ must be maximal since $P$ is rational; alternatively, $P$ is of type (b) or (c) and since $I:=P\cap R$ is a $(\sigma,\delta)$-invariant ideal and $R$ is reduced,
we see that either $R/I$ is a reduced algebra of Gelfand-Kirillov dimension two or $R/I$ is of Gelfand-Kirillov dimension $\le 1$.  In the latter case, since $T/P$ is a homomorphic image of $S:=T/IT\cong (R/I)[x;\sigma,\delta]$ and since $\sigma$ preserves a frame, we have $S$ is of Gelfand-Kirillov dimension two and thus satisfies the Dixmier-Moeglin equivalence by Proposition \ref{prop:GK2}.  It follows that $P$ is locally closed in ${\rm Spec}(T)$ in this case.

Thus it only remains to show that the Dixmier-Moeglin equivalence holds when $R/I$ is a reduced ring of Gelfand-Kirillov dimension two. Observe that since $P$ is rational we have that $T/P$ is not commutative, since otherwise it would be a field and since $T$ satisfies the Nullstellensatz, $T/P$ would have Gelfand-Kirillov dimension $0$. 
We claim that $I$ is $\sigma$-prime.  If not, $P$ falls into case (c) of Goodearl's trichotomy and so $R/I$ has a unique associated prime, which contains $(1-\sigma)(R)$.   But $I$ is radical and since $R$ is commutative all minimal primes of $R/I$ are associated primes and so we see that $I$ is in fact prime and hence $\sigma$-prime, since it is invariant under $\sigma$.

 It now suffices to show that the intersection of the prime ideals that properly contain $P$ of each type (a)--(c) properly contains $P$. The intersection of prime ideals of type (c) that contain $P$ again properly contains $P$ by Lemma \ref{lem:typec}; the intersection of primes that contain $P$ of type (a) again properly contains $P$ since it contains all commutators and $T/P$ is noncommutative by assumption.  Finally, we show that the intersection of primes that properly contain $P$ of type (b) again properly contains $P$.  By Lemma \ref{lem:R0}, $R_0=\{r\in R\colon \sigma(r)=r\}$ is equal to $k$ and since $\sigma$ preserves a frame and $\delta$ is not inner, there is some nonzero $f$ in $R$ such that $\sigma(f)=\lambda f$ with $\lambda\neq 1$. Then since $\sigma(f)-f$ is invertible in $R[1/f]$ we have $\delta$ is inner in 
$R[1/f][x;\sigma,\delta]$ and so $R[1/f][x;\sigma,\delta]\cong R[1/f][x;\sigma]$.  
By \cite[Lemma 4.14]{BSM18} we have $f$ is regular mod $I:=P\cap R$ unless $f\in I$, and so the intersection of primes $P$ of type (b) for which $f$ is not regular mod $P\cap R$ contains the element $f$.  Hence it suffices to show that the intersection of nonzero prime ideals of type (b) for which $f$ is regular mod $I:=P\cap R$ is nonzero.
But since $\sigma$ preserves a frame, $R[1/f][x;\sigma]$ satisfies the Dixmier-Moeglin equivalence \cite{BWW} and so the intersection of prime ideals $P$ of type (b) for which $f$ is regular mod $P\cap R$ is nonzero.  The result follows.
\vskip 2mm
{\bf Case II.} $P$ is a rational prime ideal of ${\rm Spec}(T)$ and $P\cap R=(0)$.
\vskip 2mm
Since $P\cap R=(0)$, $P$ survives in the localization ${\rm Frac}(R)[x;\sigma,\delta]$.  We let $\tilde{P}$ denote the ideal $P{\rm Frac}(R)[x;\sigma,\delta]$.  First, if $P$ is nonzero then $\tilde{P}$ is nonzero and hence contains a nonzero monic polynomial in $x$.  In particular, ${\rm Frac}(R)[x;\sigma,\delta]/\tilde{P}$ is a finite module over ${\rm Frac}(R)$ and hence is PI.  Thus $T/P$ is PI since ${\rm Frac}(R)[x;\sigma,\delta]/\tilde{P}$ is a localization of this ring.  But now $T/P$ is PI, it satisfies the Dixmier-Moeglin equivalence and so $P$ is locally closed in this case.  Thus we may assume that $P=(0)$.  In this case, we see as in Case I that $R$ is $\sigma$-prime, and we may argue as in Case I to show that $(0)$ is locally closed. The result follows.
\end{proof}
\begin{proof}[Proof of Theorem \ref{thm:main}] 
Suppose first that $R$ is an integral domain of Gelfand-Kirillov dimension $\ge 2$.  Since $T$ has Gelfand-Kirillov dimension less than $4$, a result of Zhang \cite{Zhang} shows that $\sigma$ preserves a frame of $R$ and that $R$ has Gelfand-Kirillov dimension two.  Since $T$ satisfies the Nullstellensatz, to prove that $T$ satisfies the Dixmier-Moeglin equivalence, it suffices to show that rational prime ideals of $T$ are locally closed in ${\rm Spec}(T)$.  Since $R$ is reduced, we have $R_K:=R\otimes_k K$ is reduced and is a finitely generated $K$-algebra whenever $K$ is a field extension of $k$, and $T_K:=T\otimes_k K\cong R_K[x;\sigma,\delta]$ is left noetherian, Jacobson, and satisfies the Nullstellensatz.  Hence to show that $T$ satisfies the Dixmier-Moeglin equivalence, using Irving-Small reduction (see \cite[Theorem 4.2.27]{Row2}), it suffices to show that $T_K$ satisfies the Dixmier-Moeglin equivalence for every algebraically closed extension of $k$.  Since $\sigma$ preserves a $K$-frame of $R_K$ and $R_K$ is reduced, $T_K$ satisfies the Dixmier-Moeglin equivalence by Proposition \ref{prop:main}, and so the result now follows in this case.

Thus we may assume that $R$ has Gelfand-Kirillov dimension zero or one.  If $R$ has Gelfand-Kirillov dimension zero, then $T$ is of Gelfand-Kirillov dimension one and hence is PI and so the result holds \cite{Von1}.  Thus we may assume that $R$ has Gelfand-Kirillov dimension one.  Then $R\otimes_k K$ is a reduced $K$ algebra whenever $K$ is an algebraically closed extension of $k$ and so the map induced by $\sigma$ preserves a frame of $R\otimes_k K$ by Lemma \ref{lem:reduced} and so using Irving-Small reduction and Proposition \ref{prop:main} again gives the result.
\end{proof}
\begin{remark} We note that the characteristic zero hypothesis is used when invoking the Irving-Small reduction techniques.
\end{remark}
\section{Errata to \cite{BWW}}
We finish this paper by taking the opportunity to list some corrections to the paper \cite{BWW}, which in many ways this paper builds upon.  The first-named author is indebted to Zhipeng Chen, who discovered the following issues.  

\begin{enumerate}
\item In the Assumptions on page 7, we should also have the assumption: all prime ideals of $R$ are completely prime and that $(0)$ is a rational prime ideal of $R[x;T]$.  In particular, Lemmas 3.2 and 3.3 should have these hypotheses (which are implicitly used in the proofs).  
\item In Lemma 3.4, there should be the assumption that $R$ satisfies the Dixmier-Moeglin equivalence. 
\end{enumerate}



\begin{thebibliography}{BBCM17}







\bibitem[ASZ99]{ASZ} M. Artin, L. W. Small, and J. J. Zhang, 
Generic flatness for strongly Noetherian algebras.
\emph{J. Algebra} {\bf 221} (1999), no. 2, 579--610.


 \bibitem[AS95]{AS} M. Artin and J. T. Stafford, Noncommutative graded domains with quadratic growth. \emph{Invent. Math.} {\bf 122} (1995), no. 2, 231--276.
 
\bibitem[Bel03]{Baffine} J. P. Bell, Examples in finite Gel'fand–Kirillov dimension. \emph{J. Algebra} {\bf 263} (2013), no. 1, 159--175.



\bibitem[Bel10]{Bell} J. P. Bell, A dichotomy result for prime algebras of Gelfand-Kirillov dimension two. \emph{J. Algebra} {\bf 324} (2010), no. 4, 831--840.

\bibitem[Bel18]{Importance}  J. P. Bell, On the importance of being primitive.  \emph{Rev. Colombiana Mat.} {\bf 53} (2019), suppl., 87--112.

\bibitem[BLLM17]{BLLM17} J. Bell, S. Launois, O. Le\'on S\'anchez, and R. Moosa, Poisson algebras via model theory and differential-algebraic geometry. \emph{J. Eur. Math. Soc. (JEMS)} {\bf 19} (2017), no. 7, 2019--2049. 


 
\bibitem[BSM18]{BSM18} J. Bell, O. Le\'on S\'anchez, and R. Moosa, $D$-groups and the Dixmier-Moeglin equivalence. \emph{Algebra Number Theory} {\bf 12} (2018), no. 2, 343--378.

\bibitem[BL14]{BL14} J. P. Bell and W. H. Leung, 
The Dixmier-Moeglin equivalence for cocommutative Hopf algebras of finite Gelfand-Kirillov dimension. 
\emph{Algebr. Represent. Theory} {\bf 17} (2014), no. 6, 1843--1852. 

\bibitem[BRS10]{BRS10} J. Bell, D. Rogalski, and S. J. Sierra, 
The Dixmier-Moeglin equivalence for twisted homogeneous coordinate rings. 
\emph{Israel J. Math.} {\bf 180} (2010), 461--507.

\bibitem[BS10]{BS} J. P. Bell and A. Smoktunowicz, 
Extended centres of finitely generated prime algebras. \emph{Comm. Algebra} {\bf 38} (2010), no. 1, 332--345.


 \bibitem[BWW17]{BWW}
 J. P. Bell, K. Wu, and S. Wu, The Dixmier-Moeglin equivalence for extensions of scalars and Ore extensions. \emph{Groups, rings, group rings, and Hopf algebras}, 1--14, \emph{Contemp. Math.}, 688, Amer. Math. Soc., Providence, RI, 2017. 
 
  
\bibitem[Ber64]{Ber} G. M. Bergman, 
A ring primitive on the right but not on the left.
\emph{Proc. Amer. Math. Soc. } {\bf15} (1964), 473--475. 





\bibitem[BG01]{BrGo} K. A. Brown and K. R. Goodearl, \emph{Lectures on algebraic quantum groups.} Advanced Courses in Mathematics. CRM Barcelona. Birkh\"auser Verlag, Basel, 2002.






\bibitem[Dix77]{Dix77}
J. Dixmier, Id\'eaux primitifs dans les alg\`ebres enveloppantes. \emph{J. Algebra} \textbf{48} (1977), 96--112.

\bibitem[Goo92]{Good} K. R. Goodearl, Prime ideals in skew polynomial rings and quantized Weyl algebras. \emph{J. Algebra} {\bf 150} (1992), no. 2, 324--377.
 



\bibitem[GL00]{GoLet} K. R. Goodearl and E. S. Letzter,  The Dixmier-Moeglin equivalence in quantum coordinate rings and quantized Weyl algebras. \emph{Trans. Amer. Math. Soc.} {\bf 352} (2000), no. 3, 1381--1403.





\bibitem[GZ10]{GZ} K. R. Goodearl and J. J. Zhang, 
Noetherian Hopf algebra domains of Gelfand-Kirillov dimension two. 
\emph{J. Algebra} {\bf 324} (2010), no. 11, 3131--3168. 



\bibitem[Irv79]{Irving} R. S. Irving,
Generic flatness and the Nullstellensatz for Ore extensions.
\emph{Comm. Algebra} {\bf 7} (1979), no. 3, 259--277.

\bibitem[IS80]{IrvSm} R. S. Irving and L. W. Small, 
On the characterization of primitive ideals in enveloping algebras.
\emph{Math. Z.} {\bf 173} (1980), no. 3, 217--221. 









\bibitem[Jou78]{Jou} J. P. Jouanolou, Hypersurfaces solutions d'une équation de Pfaff analytique. \emph{Math. Ann.} {\bf 232} (3) (1978), 239--245.



\bibitem[KL00]{KL} G. R. Krause and T. H. Lenagan, \emph{Growth of algebras and Gelfand-Kirillov dimension.} Revised edition. Graduate Studies in Mathematics, 22. American Mathematical Society, Providence, RI, 2000. 












\bibitem[Let89]{Letzter} E. Letzter, Primitive ideals in finite extensions of Noetherian rings. \emph{J. London Math. Soc.} (2) {\bf 39} (1989), no. 3, 427--435.


\bibitem[Lor77]{Lor} M. Lorenz, Primitive ideals of group algebras of supersoluble groups.
\emph{Math. Ann.} {\bf 225} (1977), no. 2, 115--122. 








\bibitem[MR01]{McR} J. McConnell and J. Robson, \emph{Noncommutative Noetherian rings.} With the cooperation of L. W. Small. Revised edition. Graduate Studies in Mathematics, 30. American Mathematical Society, Providence, RI, 2001. 


\bibitem[Moe80]{Moe80} C. Moeglin, Id\'eaux bilat\`eres des alg\`ebres enveloppantes. \emph{Bull. Soc. Math. France} {\bf 108} (1980), 143--186.






\bibitem[MS81]{SM} 
S. Montgomery and L. W. Small, 
Fixed rings of Noetherian rings.
\emph{Bull. London Math. Soc.} {\bf 13} (1981), no. 1, 33--38




\bibitem[Row88]{Row2} L. H. Rowen, \emph{Ring theory. Vol. II.} Pure and Applied Mathematics, 128. Academic Press, Inc., Boston, MA, 1988. 


\bibitem[SW84]{SW} L. W. Small and R. B. Warfield Jr., Prime affine algebras of Gelfand-Kirillov dimension one. \emph{J. Algebra} {\bf 91} (1984), no. 2, 386--389. 
\bibitem[SSW85]{SSW} L. W. Small, J. T. Stafford, and R. B. Warfield Jr.,  Affine algebras of Gel'fand-Kirillov dimension one are PI. \emph{Math. Proc. Cambridge Philos. Soc.} {\bf 97} (1985), no. 3, 407--414. 


\bibitem[Von96]{Von1} N. Vonessen, 
Actions of algebraic groups on the spectrum of rational ideals. 
\emph{J. Algebra} {\bf 182} (1996), no. 2, 383--400. 






\bibitem[Zal71]{Z} A. E. Zalesski\u\i, 
The irreducible representations of finitely generated nilpotent groups without torsion.
\emph{Mat. Zametki} {\bf 9} (1971), 199--210. 

\bibitem[Zha97]{Zhang} J. J. Zhang, A note on GK dimension of skew polynomial extensions. Proc. Amer. Math. Soc. 125 (1997), no. 2, 363–373.

\end{thebibliography}
\end{document}